\documentclass[11pt]{amsart}
\usepackage{latexsym, mathrsfs, color, tikz, multirow, bbm, mathtools}
\usepackage{graphics}
\usepackage{subcaption}
\input{xy}
\xyoption{all}

\usepackage[colorlinks=true, pdfstartview=FitV,
 linkcolor=black,citecolor=black,urlcolor=black]{hyperref}
 
\usepackage{tikz}

\usepackage{algorithm}
\usepackage{algpseudocode}

\usepackage{tikz-cd}
\usepackage{comment}
\usepackage[foot]{amsaddr}
\usepackage{orcidlink}

\DeclareFontFamily{U}{mathx}{\hyphenchar\font45}
\DeclareFontShape{U}{mathx}{m}{n}{
      <5> <6> <7> <8> <9> <10>
      <10.95> <12> <14.4> <17.28> <20.74> <24.88>
      mathx10
      }{}
\DeclareSymbolFont{mathx}{U}{mathx}{m}{n}
\DeclareFontSubstitution{U}{mathx}{m}{n}
\DeclareMathAccent{\widecheck}{0}{mathx}{"71}

\numberwithin{equation}{section}
\theoremstyle{plain}
\newtheorem{Thm}[equation]{Theorem}
\newtheorem{Prop}[equation]{Proposition}
\newtheorem{Cor}[equation]{Corollary}
\newtheorem{Lem}[equation]{Lemma}
\newtheorem{Conj}[equation]{Conjecture}

\theoremstyle{definition}
\newtheorem{Def}[equation]{Definition}

\newtheorem{Rmk}[equation]{Remark}

\newenvironment{red}{\relax\color{red}}{\relax}
\newenvironment{blue}{\relax\color{blue}}{\hspace*{.5ex}\relax}
\newenvironment{green}{\relax\color{green}}{\relax}

\newcommand{\ber}{\begin{red}}
\newcommand{\er}{\end{red}}
\newcommand{\beb}{\begin{blue}}
\newcommand{\eb}{\end{blue}}
\newcommand{\beg}{\begin{green}}
\newcommand{\eg}{\end{green}}

\newcommand{\bw}{{\boldsymbol{w}}}

\newcommand{\ured}{\text{uRed}}
\newcommand{\rred}{\text{Red}}
\DeclareMathOperator{\sgn}{sgn}
\begin{document}
	
\title[Unrestricted Red Size \& Sign-Coherence]{Unrestricted Red Size and Sign-Coherence}

\author[T. J. Ervin]{Tucker J. Ervin \orcidlink{0000-0001-5414-2379}}
\address{Department of Mathematics, University of Alabama,
	Tuscaloosa, AL 35487, U.S.A.}
\email{tjervin@crimson.ua.edu}

\begin{abstract}
    The unrestricted red size of a quiver is the maximal number of red vertices in its framed quiver after any given mutation sequence.
    In a 2023 paper by E. Bucher and J. Machacek, it was shown that connected, mutation-finite quivers either have an unrestricted red size of $n-1$ or $n$, where $n$ is the number of vertices in the quiver.
    We prove here that the same holds for the connected, mutation-infinite case using forks.
    As such, the unrestricted red size for any quiver equals $n-c$, where $c$ is the number of connected components of the quiver that do not admit a reddening sequence.
    Additionally, we prove a result on the $c$-vectors of forks that allows us to show that the $c$-vectors of both abundant acyclic quivers on any number of vertices and mutation-cyclic quivers on three vertices are sign-coherent with only elementary methods. 
\end{abstract}

\maketitle

\section{Introduction} \label{sec-introduction}

Maximal green sequences were formally dubbed such in a 2010 talk by B. Keller, which was later compiled into a series of lecture notes \cite{keller_cluster_2011}.
Both maximal green sequences and reddening sequences are special sequences of mutations that end in quivers with all red vertices, and they have powerful implications for the cluster algebras and Donald-Thompson invariants associated to quivers that admit them \cite{demonet_survey_2020}. 
Not every quiver admits a maximal green sequence or a reddening sequence, and in 2020 Ahmad and Li created a measure, the red size $\rred$, which counts the maximal number of red vertices in a quiver after mutating only at a green vertices \cite{ahmad_orbit-maximal_2020}.
In 2023, Bucher and Machacek generalized this to reddening sequences with the unrestricted red size $\ured$, which counts the maximal number of red vertices after any given mutation sequence \cite{bucher_red_2023}. 
Naturally, for a quiver $Q$ on $n$ vertices that admits a reddening sequence, we have that $\ured(Q) = n$.
The question then becomes: what is the unrestricted red size of $Q$ when it does not admit a reddening sequence?

Bucher and Machacek conjectured the following for the case of connected quivers. 

\begin{Conj}{\cite{bucher_red_2023}} \label{conj-b-m-connected}
    For any connected quiver $Q$ on $n$ vertices which does not admit a reddening sequence
    \begin{enumerate}
        \item We have $\rred(Q) = n-1$.

        \item We have $\rred(Q) = n-1$ and every vertex of $Q$ is the last remaining green vertex after the application of some general maximal green sequence. 

        \item We have $\ured(Q) = n-1$.

        \item We have $\ured(Q) = n-1$ and every vertex of $Q$ is the last remaining green vertex after the application of some general reddening sequence. 
    \end{enumerate}
\end{Conj}

Answering any part of Conjecture \ref{conj-b-m-connected} affirmatively would naturally imply their other conjecture ---a generalization of the mutation-invariance of admitting a reddening sequence.

\begin{Conj}{\cite{bucher_red_2023}} \label{conj-b-m-invariant}
    The unrestricted red size is a mutation-invariant.
\end{Conj}

Conjecture \ref{conj-b-m-connected} was proven for the connected, mutation-finite case.
Using a special type of quiver called a fork, we extend their results to the connected, mutation-infinite case.
This proves part (3) of Conjecture \ref{conj-b-m-connected}; however, the stronger variations of Conjecture \ref{conj-b-m-connected} remain unresolved.

\begin{Cor} \label{cor-intro-tsc-fork-ured-size}
    Every connected, mutation-infinite quiver $Q$ on $n$ vertices has $uRed(Q) \geq n-1$.
    Furthermore, if $Q$ is any quiver on $n$ vertices, then $\ured(Q) = n - c$, where $c$ is the number of connected components of $Q$ that do not admit reddening sequences.
\end{Cor}

To tackle this question, we explore a generalization of sign-coherence.
Instead of depending on the signs of the $c$-vectors arising from principal coefficients, we allow a broader class of starting coefficients.
Those ice quivers which still have their $c$-vectors satisfy the sign-coherence property under any mutation sequence are of particular interest.
We take the name \textit{uniformly sign-coherent} for these ice quivers from a paper by Cao and Li \cite{cao_uniform_2019}.

As sign-coherence for $c$-vectors is already a well-known and well-studied property of framed quivers \cite{derksen_quivers_2010, gross_canonical_2018} and all connected and mutation-infinite quivers are mutation-equivalent to a fork \cite{warkentin_exchange_2014}, we demonstrate our result on the unrestricted red size by finding a sequence ending in $n-1$ red vertices for arbitrary \textit{strictly sign-coherent}---which disallows the zero vector as a $c$-vector---forks with $n$ mutable vertices. 
While exploring this topic, we found Proposition \ref{prop-control-signs-fork}: a way to describe the arrows between frozen and mutable vertices after mutating in a fork-preserving way.
This proposition has applications in a future work of the author and his collaborators, whose extended abstract was recently showcased at the 2024 FPSAC conference \cite{ervin_geometry_2024}.
Additionally, this argument can be modified slightly, culminating in Corollaries \ref{cor-sign-coherence-abundant acyclic} and \ref{cor-sign-coherence-mutation-cyclic-3}.
The two corollaries prove that the $c$-vectors of both abundant acyclic quivers on any number of vertices and mutation-cyclic quivers on three vertices are sign-coherent with only elementary methods. 
Similarly, we expand slightly on Proposition \ref{prop-control-signs-fork} to show that mutation-cyclic quivers on three vertices do not admit reddening sequences with only elementary methods.

As for the structure of the paper, we list the necessary preliminary material in Section \ref{sec-prelim}. 
Section \ref{sec-acyclic} develops the information on acyclic orderings that we need for Section \ref{sec-sc}, where we prove our main result.
Finally, Section \ref{sec-sign-coherent-forks} gives some results on the general $c$-vectors of ice quivers.

\subsection{Acknowledgements}

I would like to thank Eric Bucher, John Machacek, and Scott Neville for their discussions about this topic. 
I would also like to thank my advisor, Kyungyong Lee, for his guidance, support, and help revising this paper.

\section{Preliminaries} \label{sec-prelim}

\subsection{Quivers \& Mutation}

We begin with some definitions concerning quivers and quiver mutation.

\begin{Def} \label{def-quivers}
    A quiver $Q$ is defined as a tuple $(Q_0, Q_1, s,t)$, where $Q_0$ is a set of vertices, $Q_1$ is a set of arrows, and maps $s,t: Q_1 \to Q_0$ taking any arrow $a \in Q_1$ to its starting vertex $s(a)$ and its terminal vertex $t(a)$, i.e., $s(a) \to t(a)$ is the arrow $a$ in $Q$.
    We further restrict our definition so that $s(a) \neq t(a)$ (no loops) for all $a \in Q_1$ and that there is no pair of arrows $a,b \in Q_1$ where $s(a) = t(b)$ and $t(a) = s(b)$ (no 2-cycles).
    This naturally forces all arrows between two vertices to point in the same direction.
    For example, Figure \ref{fig-distinct-quivers} demonstrates two distinct quivers.
    
    \begin{figure}[ht]
        \centering
        \[\begin{tikzcd}
        i & j 
        \arrow[from=1-1,to=1-2]
        \end{tikzcd}
        \text{ and }\begin{tikzcd}
        i & j 
        \arrow[from=1-2,to=1-1]
        \end{tikzcd}\]
        \caption{Two Distinct Quivers}
        \label{fig-distinct-quivers}
    \end{figure}
    
    Additionally, whenever we have multiple arrows between the same two vertices, we write the multiplicity, say $m$, above or below the arrow, i.e, as in Figure \ref{fig-multiplicity}.
    
    \begin{figure}[ht]
        \centering
        \begin{tikzcd}
        i & j 
        \arrow[from=1-1,to=1-2, "m"]
        \end{tikzcd}
        \caption{Notation for Arrows With Multiplicity $m$}
        \label{fig-multiplicity}
    \end{figure}
    \begin{itemize}
        \item A vertex that only has arrows leaving (entering) is referred to as a \textit{source} (\textit{sink}).

        \item A quiver is called \textit{acyclic} if there exists no (nonempty) directed path of arrows beginning and ending at the same vertex.

        \item A quiver is said to be \textit{abundant} if there are at least two arrows between every pair of vertices. 

        \item A quiver is \textit{connected} if there exists a sequence of edges (considered without their directions) connecting any two pairs of vertices.

        \item A quiver $Q = (Q_0, Q_1, s,t)$ is called a \textit{full subquiver} (often called an \textit{induced subquiver} or just \textit{subquiver} if the context is clear) of a quiver $P = (P_0,P_1,s',t')$ if $Q_0 \subset P_0$, if $Q_1 \subseteq P_1$, if the maps $s$ and $t$ are restrictions of $s'$ and $t'$ to the set $Q_1$, and if $s(a), t(a) \in Q_0$ implies that $a \in Q_1$ for any arrow $a \in P_1$.
        For ease of use, the notation $P \setminus V$ will refer to the full subquiver induced by the vertex set $P_0 \setminus V$ for some set $V$.
    \end{itemize}
    Fix a quiver, denoted by a capital letter, say $Q$.
    Then the number of arrows between vertex $i$ and $j$ in $Q$ is denoted by the lowercase letter indexed by the vertices, $q_{ij}$, which is positive if $i \to j$ and negative if $j \to i$.
\end{Def}

\begin{Def} \label{def-mutation}
    \textit{Quiver mutation} is defined for every quiver $Q$ and vertex $v$ in $Q_0$ by the following:
    \begin{itemize}
        \item For every path of the form $a \to v \to b$ in $Q$, add an arrow from $a$ to $b$;
        
        \item Flip the direction of every arrow touching $v$;
        
        \item Finally, if any 2-cycles were formed (paths of length two that begin and end at the same vertex), remove them.
    \end{itemize}
   This process will always mutate a quiver into another quiver, and we denote the mutation of $Q$ at vertex $v$ by $\mu_v(Q)$.
   If we have a sequence of vertices, $\bw = [i_1,i_2,i_3,\dots,i_n]$, then we use 
   $$\mu_\bw(Q) = \mu_{i_n}(\mu_{[i_1,i_2,\dots,i_{n-1}]}(Q))$$
   to denote the result of mutating $Q$ at each vertex $i_j$ in order.
   We call $\bw$ a \textit{mutation sequence}.
   If we are always mutating at sources, then we call $\bw$ a \textit{source mutation sequence}.
   Additionally, we assume that all of our mutation sequences are \textit{reduced}, i.e., we never see mutation at a vertex and then another immediate mutation at that same vertex.
\end{Def}

\begin{Def} \label{def-mutation-class}
    Two quivers are said to be \textit{mutation-equivalent} if there exists a sequence of mutations taking one to the other.
    The collection of all quivers mutation-equivalent to a quiver $Q$ is denoted $[Q]$, and it is called the \textit{labelled mutation class} of $Q$.
    If $[Q]$ is a finite (infinite) set, then $Q$ is said to be \textit{mutation-finite} (\textit{mutation-infinite}).
    Further, if $Q$ is mutation-equivalent to an acyclic quiver, it is said to be \textit{mutation-acyclic}.
    If $Q$ is not mutation-acyclic, it is \textit{mutation-cyclic}.
\end{Def}

Which vertices we allow ourselves to mutate at separate quivers from ice quivers. 
The former allow mutation at any vertex.
The latter have certain frozen vertices for which mutation is forbidden.

\begin{Def} \label{def-red-green-blue}
    A vertex of a quiver is said to be \textit{frozen} if we do not allow mutation at that vertex.
    Any non-frozen vertex is \textit{mutable}, and quivers with frozen vertices are called \textit{ice quivers} when a distinction needs to be made.
    A mutable vertex $i$ in a quiver $Q$ is said to be \textit{red} (\textit{green}) if $|q_{ji}| \neq 0$ implies that $q_{ji} > 0$ ($q_{ij} > 0$) for all frozen vertices $j$.
    If $q_{ji} = 0$ for all frozen vertices $j$, we say that $i$ is \textit{blue}.
\end{Def}

An important class of ice quivers are the framed quivers, which are the objects of study that Conjecture \ref{conj-b-m-connected} is concerned with.

\begin{Def} \label{def-framed-quiver}
    We may form a \textit{framed} quiver $\widehat{Q}$ from a quiver $Q$ in the following manner.
    Add a frozen vertex $i'$ to $Q$ for each vertex $i$, where a single arrow points from $i$ to $i'$.
    For example, if $Q$ was the quiver $i \to j$, then the framed  quiver $\widehat{Q}$ would be given in Figure \ref{fig-framed-quiver}.
    
    \begin{figure}[ht]
        \centering
        \begin{tikzcd}
            i' & j' \\
            i & j 
            \arrow[from=2-1,to=1-1]
            \arrow[from=2-1,to=2-2]
            \arrow[from=2-2,to=1-2]
        \end{tikzcd}
        \caption{Example of Framed Quiver}
        \label{fig-framed-quiver}
    \end{figure}
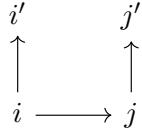
    
    Similarly, we may form the \textit{coframed quiver} $\widecheck{Q}$ by adding a frozen vertex $i'$ to $Q$ for each vertex $i$, where a single arrow points from $i'$ to $i$.
    For example, if $Q$ was $i \to j$, then the coframed $\widecheck{Q}$ would be given in Figure \ref{fig-coframed-quiver}.
    
    \begin{figure}[ht]
        \centering
        \begin{tikzcd}
            i' & j' \\
            i & j 
            \arrow[from=1-1,to=2-1]
            \arrow[from=2-1,to=2-2]
            \arrow[from=1-2,to=2-2]
        \end{tikzcd}
        \caption{Example of Coframed Quiver}
        \label{fig-coframed-quiver}
    \end{figure}
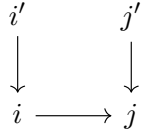
\end{Def}

The reason why we focus on framed and coframed quivers instead of truly arbitrary ice quivers comes from the sign-coherence theorem, which controls the colors of vertices in the labelled mutation class of $\widehat{Q}$ and $\widecheck{Q}$.

\begin{Thm}{ \cite{derksen_quivers_2010}} \label{thm-sign-coherence}
    Let $Q$ be any quiver.
    If $P$ is mutation-equivalent to $\widehat{Q}$, then every mutable vertex of $P$ is either red or green.
    Similarly, if $P$ is mutation-equivalent to $\widecheck{Q}$, then every mutable vertex of $P$ is either red or green.
\end{Thm}

We can then define the mutation sequences that we wish to study.

\begin{Def} \label{def-red-sequence}
    A mutation sequence $\bw$ is said to be a \textit{reddening sequence} for an ice quiver $Q$ if $P = \mu_\bw(Q)$ has $n$ red vertices, where $n$ is the number of mutable vertices in $Q$.
    If we only mutated at green vertices along $\bw$, the sequence is instead called a \textit{maximal green sequence}.
    Bucher and Machacek generalized reddening sequences to \textit{general reddening sequences}, which requires that $P$ have a maximal number, $uRed(Q)$, of red vertices.
    We say that $uRed(Q)$ is the unrestricted red size of $Q$ \cite{bucher_red_2023}.

    If $Q$ has only mutable vertices, we will say that $Q$ has a reddening sequence, maximal green sequence, a general reddening sequence, etc. whenever the framed quiver, $\widehat{Q}$, does.
\end{Def}

\subsection{Matrices}

In Section \ref{sec-sign-coherent-forks}, it will be advantageous to use the exchange matrix of a quiver.

\begin{Def}{\cite{fomin_introduction_2021-1}} \label{def-extended-skew-symmetrizable}
    Let $Q$ be any ice quiver with $n$ mutable vertices and $m$ frozen vertices.
    Furthermore, assume that the mutable vertices are labelled $1$ through $n$ and that the frozen vertices are labelled $n+1$ through $n+m$.
    Then we may construct a matrix $[ B | C]$, known as the \textit{extended exchange matrix}, where $B$ is a $n$-by-$n$ matrix and $C$ is a $n$-by-$m$ matrix.
    The matrix $B$ is constructed by letting $b_{ij} = q_{ij}$ for all mutable vertices $i$ and $j$.
    The matrix $C$ is constructed by letting $c_{ij} = q_{ij}$ for all mutable vertices $i$ and frozen vertices $j$.
\end{Def}

Naturally, this process agrees with mutation, giving rise to the well-known matrix version of mutation.

\begin{Def}{\cite{fomin_introduction_2021-1}} \label{def-matrix-mutation}
    Let $[ B | C]$ be a matrix corresponding to an ice quiver $Q$.
    Then the mutation of $[B | C]$, denoted $[B^{[k]} | C^{[k]}]$ at a mutable vertex of $Q$, $k$, is given by 
    $$b_{ij}^{[k]} = \begin{cases}
        b_{ij} + \sgn(b_{ik}) b_{ik}b_{kj} = b_{ij} + |b_{ik}|b_{kj} & \text{if } b_{ik}b_{kj} > 0\\
        -b_{ij} & \text{if } k \in \{i,j\}\\
        b_{ij} & \text{otherwise}
    \end{cases}$$
    and
    $$c_{ij}^{[k]} = \begin{cases}
        c_{ij} + \sgn(b_{ik}) b_{ik}c_{kj} = c_{ij} + |b_{ik}|c_{kj}& \text{if } b_{ik}c_{kj} > 0\\
        -c_{ij} & \text{if } k = i\\
        c_{ij} & \text{otherwise}
    \end{cases}.$$
    If $\bw$ is a mutation sequence, then we can iterate this process to find $B^\bw$ and $C^\bw$, the matrices corresponding to $\mu_\bw(Q)$.
    The $i$th row of $C^\bw$ is denoted $c_i^\bw$ and is called a $c$-vector.
\end{Def}

\begin{Rmk} \label{rmk-red-green-c-vector}
    In terms of our earlier notation, we see that a mutable vertex, $v$, of $Q$ is green if and only if $c_{vj} > 0$ for some frozen vertex $j$ and $c_{vk} \geq 0$ for every frozen vertex $k \neq j$.
    We then say that $c_{v}$, the corresponding row vector of the $C$-matrix, has positive sign, equivalently $\sgn(c_v) > 0$ or $c_v \geq 0$ (this last inequality is taken component wise).
    Similarly, a vertex $v$ of $Q$ is red if and only if $c_{vj} < 0$ for some frozen vertex $j$ and $c_{vk} \leq 0$ for every other frozen vertex $k$.
    We then say that $c_{v}$ has negative sign, equivalently $\sgn(c_v) < 0$ or $c_v \leq 0$.
    A vertex $v$ of $Q$ is blue if and only if $c_{v} = 0$.
\end{Rmk}

Theorem \ref{thm-sign-coherence} was originally given in the language of $g$-vectors in the paper by Derksen et al \cite{derksen_quivers_2010}.
However, it easily translates to $c$-vectors and is proven in more generality by Gross et al \cite{gross_canonical_2018}, where sign-coherence refers to all the components of a $c$-vector having the same sign.
For our purposes, we will be thinking of sign-coherence in the terms of arbitrary ice quivers, arising from ice quivers which are not necessarily a framing.
Before we get to this generalization, we need to introduce forks.

\subsection{Forks}

Forks were first introduced in a 2014 dissertation by M. Warkentin \cite{warkentin_exchange_2014}, and they form the backbone of our two main results.

\begin{Def}{\cite[Definition 2.1]{warkentin_exchange_2014}} \label{def-forks}
A \textit{fork} is an abundant quiver $F$, where $F$ is not acyclic and where there exists a vertex $r$, called the point of return, such that
\begin{itemize}
    \item For all $i \in F^{-}(r)$ and $j \in F^{+}(r)$ we have $f_{ji} > f_{ir}$ and $f_{ji} > f_{rj}$, where $F^{-}(r)$ is the set of vertices with arrows pointing towards $r$ and $F^{+}(r)$ is the set of vertices with arrows coming from $r$.
    
    \item $F \setminus \{r\}$ is an acyclic quiver.
\end{itemize}
An example of a fork is given in Figure \ref{fig-fork-example}, 
where $r$ is the point of return.
\begin{figure}[ht]
    \centering
    \begin{tikzcd}
    & r \\
    i & & j
    \arrow[from=2-1, to=1-2, "3"]
    \arrow[from=1-2, to=2-3, "4"]
    \arrow[from=2-3, to=2-1, "5"]
\end{tikzcd}
    \caption{Example of a Fork With Point of Return $r$}
    \label{fig-fork-example}
\end{figure}
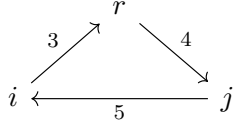
\end{Def}

All of the useful properties arising from forks have been proven using elementary methods by Warkentin.
Of these, Lemma \ref{lem-fork-mutation} and Theorem \ref{thm-connected-fork} are some of the most important.
Lemma \ref{lem-fork-mutation} gives forks their name, as in practice it means any mutation sequence reaching a fork reaches a "fork in the road."
Theorem \ref{thm-connected-fork} is what allows us to prove our main result on the unrestricted red size.

\begin{Lem}{\textup{\cite[Lemma 2.5]{warkentin_exchange_2014}}} \label{lem-fork-mutation}
    Let $F$ be either a fork with point of return $r$ or an abundant acyclic quiver.
    If $k$ is a vertex in $F$ that is not the point of return, not a source, or not a sink, then $\mu_k(F)$ is a fork with point of return $k$.
    Additionally, the total number of arrows contained in $\mu_k(F)$ is strictly greater than the total number of arrows in $F$. 
\end{Lem}

\begin{Thm}{\textup{\cite[Theorem 3.2]{warkentin_exchange_2014}}} \label{thm-connected-fork}
    A connected quiver is mutation-infinite if and only if it is mutation-equivalent to a fork.
\end{Thm}

\section{Acyclic Orderings} \label{sec-acyclic}

In order to keep track of the color of vertices in a fork, we will first need some results on the acyclic orderings arising from abundant acyclic quivers and forks.

\begin{Def}{\textup{\cite{bang-jensen_classes_2018}}} \label{def-acyclic-ordering}
    Let $Q$ be a quiver on $n$ vertices.
    Then we say that an ordering $v_1 \prec v_2 \prec \dots \prec v_n$ on the vertices is acyclic if whenever there exists an arrow $v_i \to v_j$ then $v_i \prec v_j$.
\end{Def}

\begin{Prop}{\textup{\cite[Proposition 3.1.2]{bang-jensen_classes_2018}}} \label{prop-acyclic-ordering}
    Let $Q$ be an acyclic quiver.
    Then there exists an acyclic ordering of its vertices.
\end{Prop}

Naturally, if a quiver is abundant and acyclic, its acyclic ordering must be unique.
We can then track the acyclic orderings of every acyclic quiver in its source mutation cycle.

\begin{Lem} \label{lem-source-acyclic-mutation-order}
    Let $Q$ be an abundant acyclic quiver, where $v_1 \prec v_2 \prec \dots \prec v_n$ is its unique acyclic ordering.
    Then the mutation sequence $\bw_j = [v_1, v_2, \dots, v_j]$ is a source mutation sequence such that $Q^j = \mu_{\bw_j}(Q)$ has the unique acyclic ordering 
    $$v_{j+1} \prec  \dots \prec v_n \prec v_1 \prec \dots \prec v_{j-1} \prec v_{j}$$ 
    for all $j \leq n$.
\end{Lem}

\begin{proof}
    Mutating $Q$ at $v_1$ will give us the quiver $Q^1$ with unique acyclic ordering $v_2 \prec v_3 \prec \dots \prec v_n \prec v_1$, as the source of $Q$ will become the sink of $Q^1$, affecting no other arrows.
    As such, if we keep mutating along $\bw_i$, we cycle through the acyclic ordering.
    Thus $Q^i = \mu_{\bw_i}(Q)$ has the unique acyclic ordering $v_{i+1} \prec v_{i+2} \prec \dots \prec v_n \prec v_1 \prec \dots \prec v_{i-1} \prec v_{i}$ for all $i < n$.
    If $i = n$, then we have arrived back at our original $Q$.
\end{proof}

As removing the point of a return from a fork produces an abundant acyclic quiver, we can also examine the acyclic ordering produced by subquivers of a fork after mutation not at the point of return.

\begin{Lem} \label{lem-fork-mutate-order}
    Let $F$ be a fork with point of return $r$, where $v_1 \prec v_2 \prec \dots \prec v_{n-1}$ is the unique acyclic ordering on $F \setminus \{r\}$.
    For some vertex $v_j$, let $F' = \mu_{v_j}(F)$.
    Then $F' \setminus \{v_j\}$ has unique acyclic ordering 
    $$r \prec v_1 \prec \dots \prec v_{j-1} \prec v_{j+1} \prec \dots \prec v_{n-1}$$ 
    if $r \to v_j$ in $F$.
    If instead $v_j \to r$ in $F$, we have the unique acyclic ordering
    $$ v_1 \prec \dots \prec v_{j-1} \prec v_{j+1} \prec \dots \prec v_{n-1} \prec r$$
\end{Lem}

\begin{proof}
    First, note that the direction of arrows in $F \setminus \{r,v_{j}\}$ is completely unchanged after mutation at $v_{j}$, as $F \setminus \{r\}$ is abundant acyclic.
    Thus the unique acyclic ordering of $F' \setminus \{r,v_{j}\}$ is $v_1 \prec \dots \prec v_{j-1} \prec v_{j+1} \prec \dots \prec v_{n-1}$.
    As $F' \setminus \{v_j\}$ is abundant acyclic, we need only add $r$ into the ordering.

    Before doing so, note that $r \to v_1$ and $v_n \to r$ in $F$, as forks do not have sources or sinks.
    Now suppose that $r \to v_j$ in $F$.
    Then $r \to v_1$ in $F'$, and the unique acyclic ordering of $F' \setminus \{v_j\}$ is $r \prec v_1 \prec \dots \prec v_{j-1} \prec v_{j+1} \prec \dots \prec v_{n-1}$.
    If instead $v_j \to r$ in $F$, then $v_n \to r$ in $F'$, and the unique acyclic ordering of $F' \setminus \{v_j\}$ is $ v_1 \prec \dots \prec v_{j-1} \prec v_{j+1} \prec \dots \prec v_{n-1} \prec r$.
\end{proof}

Lemmata \ref{lem-source-acyclic-mutation-order} and \ref{lem-fork-mutate-order} combine to give us a thorough understanding of the two acyclic orderings that can be found by performing source mutations on $F \setminus \{r\}$.

\begin{Lem} \label{lem-source-fork-mutation-order}
    Let $F$ be a fork with point of return $r$, where $v_1 \prec v_2 \prec \dots \prec v_{n-1}$ is the unique acyclic ordering on $F \setminus \{r\}$.
    If $\bw_j = [v_1, v_2, \dots, v_j]$ for $j \leq n-1$, let $F^j = \mu_{\bw_j}(F)$.
    Then $F^j \setminus \{v_j\}$ has unique acyclic ordering 
    $$r \prec v_{j+1} \prec \dots \prec v_{n-1} \prec v_{1} \prec \dots \prec v_{j-1}$$
    and $F^j \setminus \{r\}$ has unique acyclic ordering 
    $$v_{j+1} \prec  \dots \prec v_{n-1} \prec v_1 \prec \dots \prec v_{j-1} \prec v_{j}.$$
\end{Lem}

\begin{proof}
    Lemma \ref{lem-source-acyclic-mutation-order} proves the desired result for the acyclic ordering on $F^j \setminus \{r\}$, as $F \setminus \{r\}$ is abundant acyclic and $\bw_j$ is a source mutation sequence.
    We will then show that $r$ is the source of $F^j \setminus \{v_j\}$ by induction.
    As $r \to v_1$ in $F$, we know that $F^1 \setminus \{v_1\}$ must have acyclic ordering $r \prec v_2 \prec v_{3} \prec \dots \prec v_{n-1}$ by Lemma \ref{lem-fork-mutate-order}, completing the base case.

    Assume then that $F^j \setminus \{v_j\}$ has acyclic ordering $r \prec v_{j+1} \prec \dots \prec v_{n-1} \prec v_{1} \prec \dots \prec v_{j-1}$.
    As $F^j \setminus \{r\}$ has unique acyclic ordering $v_{j+1} \prec  \dots \prec v_{n-1} \prec v_1 \prec \dots \prec v_{j-1} \prec v_{j}$, this forces $v_{j+1} \to v_{j}$ in $F^j$.
    Thus, Lemma \ref{lem-fork-mutate-order} tells us that $F^{j+1} \setminus \{v_{j+1}\}$ has unique acyclic ordering $r \prec v_{j+2} \prec \dots \prec v_{n-1} \prec v_{1} \prec \dots \prec v_{j-1} \prec v_{j}$, completing the inductive case.
\end{proof}

\begin{Cor} \label{cor-source-r-fork-mutation-order}
    Let $F$ be a fork with point of return $r$, where $v_1 \prec v_2 \prec \dots \prec v_{n-1}$ is the unique acyclic ordering on $F \setminus \{r\}$.
    If $\bw = [v_1, v_2, \dots, v_j, r]$ for $j \leq n-1$, let $F' = \mu_{\bw}(F)$. 
    Then $F' \setminus \{v_j\}$ has unique acyclic ordering 
    $$v_{j+1} \prec \dots \prec v_{n-1} \prec v_{1} \prec \dots \prec v_{j-1} \prec r$$ 
    and $F' \setminus \{r\}$ has unique acyclic ordering 
    $$v_{j} \prec v_{j+1} \prec  \dots \prec v_{n-1} \prec v_1 \prec \dots \prec v_{j-1}$$
\end{Cor}

\section{Sign-coherent Ice Quivers} \label{sec-sc}

We now give a definition of two different types of quivers with sign-coherent $c$-vectors, where the name uniformly sign-coherent comes from the term ``uniform column sign-coherent'' given by Cao and Li \cite{cao_uniform_2019}.

\begin{Def} \label{def-usc-quiver}
    An ice quiver is said to be \textit{uniformly sign-coherent} if every quiver in its mutation class has every vertex red, green, or blue.
    Equivalently, every $c$-vector has either negative sign, has positive sign, or is the zero vector.
    In this paper, we assume that not every vertex of a uniformly sign-coherent quiver is blue.
    If we do use such a quiver, we will refer to it as a \textit{trivially sign-coherent} quiver.
    Finally, if every quiver in a mutation class has every vertex exclusively red or green, we say it is \textit{strictly sign-coherent}.
    Equivalently, every $c$-vector has either negative or positive sign.
\end{Def}

Introducing strictly sign-coherent quivers allows us to speak about ice quivers that are mutation equivalent to a framed or coframed quiver.
The sign-coherence theorem forces all such quivers to be strictly sign-coherent.
We will prove some of our results for the more general uniformly sign-coherent case, but the strictly sign-coherent case is all we need, beginning with a result that handles how colors change after mutation in relation to the acyclic ordering of an abundant acyclic quiver.

\begin{Lem} \label{lem-color-abundant acyclic}
    Let $Q$ be an abundant acyclic and uniformly sign-coherent quiver, where $v_1 \prec v_2 \prec \dots \prec v_n$ is its unique acyclic ordering.
    Pick a vertex $v_j$ of $Q$ for $j \in Q_0$, and let $Q' = \mu_{v_j} (Q)$.
    \begin{itemize}
        \item If $v_j$ is blue in $Q$, then the color of every vertex is unchanged after mutation.

        \item If $v_j$ is green in $Q$, then the color of every vertex $v_i$ of $Q$ such that $v_j \prec v_i$ is unchanged after mutation and $v_j$ is red in $Q'$.

        \item If $v_j$ is red in $Q$, then the color of every vertex $v_i$ of $Q$ such that $v_i \prec v_j$ is unchanged after mutation and $v_j$ is green in $Q'$.
    \end{itemize}
\end{Lem}

\begin{proof}
    If $v_j$ is blue in $Q$, then there are no arrows between $v_j$ and frozen vertices in $Q$.
    As such, mutation at $v_j$ will neither add or subtract arrows between mutable and frozen vertices in $Q'$.
    Thus, the color of every vertex in $Q$ is unchanged after mutation at $v_j$.

    If $v_j$ is green in $Q$, then there are arrows going from $v_j$ to frozen vertices.
    Thus, for any vertex $v_i$ of $Q$ such that $v_j \prec v_i$, we know that there is no path of length 2 from a frozen vertex to $v_i$ or a path of length 2 from $v_i$ to a frozen vertex passing through $v_j$.
    This follows from the fact that $v_j \prec v_i$ if and only if there exists and arrow $v_j \to v_i$.
    Naturally, this means that color of the vertices $v_i$ such that $v_j \prec v_i$ is unchanged after mutation at $v_j$, and the color of $v_j$ must be red in $Q'$.
    A similar argument proves our result for the case where $v_j$ is red in $Q$.
\end{proof}

From Lemma \ref{lem-color-abundant acyclic}, we have an immediate extension to forks.

\begin{Cor} \label{cor-color-fork}
    Let $F$ be a uniformly sign-coherent fork with point of return $r$, where $v_1 \prec v_2 \prec \dots \prec v_{n-1}$ is the unique acyclic ordering on $F \setminus \{r\}$.
    Pick a vertex $v_j$ of $F \setminus \{r\}$ for $j \in [n-1]$, and let $F' = \mu_{v_j} (F)$.
    \begin{itemize}
        \item If $v_j$ is blue in $F$, then the color of every vertex is unchanged after mutation.

        \item If $v_j$ is green in $F$, then the color of every vertex $v_i$ of $F \setminus \{r\}$ such that $v_j \prec v_i$ is unchanged after mutation, the color of $r$ is unchanged if $v_j \in F^-(r)$,  and $v_j$ is red in $F'$.

        \item If $v_j$ is red in $F$, then the color of every vertex $v_i$ of $F \setminus \{r\}$ such that $v_i \prec v_j$ is unchanged after mutation, the color of $r$ is unchanged if $v_j \in F^+(r)$, and $v_j$ is green in $F'$.
    \end{itemize}
\end{Cor}

In the same vein of thought, we can control the colors of vertices after a source mutation sequence in an abundant acyclic quiver. 

\begin{Lem} \label{lem-color-source-acyclic-mutation-order}
    Let $Q$ be an abundant acyclic and uniformly sign-coherent quiver, where $v_1 \prec v_2 \prec \dots \prec v_n$ is its unique acyclic ordering.
    Pick a vertex $v_j$ of $Q$ for $j \in Q_0$, and let $Q^{j} = \mu_{\bw_j}(Q)$ for $\bw_j = [v_1, v_2, \dots, v_j]$.
    If $v_j$ is red in $Q$, then $v_1, v_2, \dots, v_{j-1}$ are red vertices and $v_j$ is a green vertex in $Q^{j}$.
\end{Lem}

\begin{proof}
    As $\bw_i$ is a source mutation sequence for all $i$ such that $1 \leq i < j \leq n$, we know that $v_i$ is a source in $Q^{i-1}$, where $Q^0 = Q$.
    If $v_i$ is green or blue in $Q^{i-1}$, Lemma \ref{lem-color-abundant acyclic} tells us that $v_i$ is red or blue respectively in $Q^i$ and every other vertex's color remains unchanged.
    If $v_i$ is red in $Q^{i-1}$, then mutation at $v_i$ adds arrows from frozen vertices to every mutable vertex, as $v_i$ is the source of $Q^{i-1}$.
    As such, all blue vertices in $Q^{i-1}$ become red vertices in $Q^{i}$.
    Let $w$ be a frozen vertex touching $v_i$ in $Q^{i-1}$.
    If $v_k$ is a red vertex of $Q^{i-1}$ for some $k \neq i$, then the subquiver given in Figure \ref{fig-subquiver-Q-i-1} exists in $Q^{i-1}$ for $a \geq 2$, $b \geq 1$, and $c \geq 0$.
    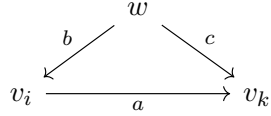
\begin{figure}[ht]
        \centering
        \begin{tikzcd}
             & w\\
            v_i & & v_k
            \arrow[from=1-2,to=2-1, "b"']
            \arrow[from=2-1,to=2-3, "a"']
            \arrow[from=1-2,to=2-3, "c"]
        \end{tikzcd}
        \caption{Subquiver of $Q^{i-1}$ Induced by $\{v_i, v_k, w\}$}
        \label{fig-subquiver-Q-i-1}
    \end{figure}
    
    After mutation at $i$, we have $c+ab$ arrows between $w$ and $v_k$, where $c+ab > b$ as $a \geq 2$.
    This shows that any red vertices in $Q^{i-1}$ other than $v_i$ remain red in $Q^i$ and that the number of arrows between a frozen vertex and red vertex in $Q^{i-1}$ increases after mutation at a red $v_i$.
    Applying this concept to $v_j$ shows that $v_j$ is red in $Q^{j-1}$.
    Additionally, if $v_i$ is green in $Q^{j-1}$ for any $i < j$ and $w$ is a frozen vertex touching $v_i$, then $v_i$ was a red vertex touching $w$ in $Q^{i-1}$.
    Our previous argument---substituting $v_j$ for $v_k$---then shows that the quiver given in Figure \ref{fig-subquiver-Q-j-1} is a subquiver of $Q^{j-1}$ where $x \geq 2$ and $y < z$.
    Thus mutating at $v_j$ forces $v_i$ to be red in $Q^j$.
    The result then follows.
    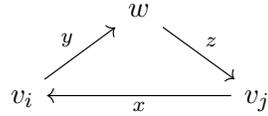
\begin{figure}[ht]
        \centering
        \begin{tikzcd}
            & w\\
            v_i & & v_j
            \arrow[from=2-1,to=1-2, "y"]
            \arrow[from=2-3,to=2-1, "x"]
            \arrow[from=1-2,to=2-3, "z"]
        \end{tikzcd}
        \caption{Subquiver of $Q^{j-1}$ Induced by $\{v_i, v_j, w\}$}
        \label{fig-subquiver-Q-j-1}
    \end{figure}
\end{proof}

If the sink of $Q$ is red, then we get the following corollary.

\begin{Cor} \label{cor-source-acyclic-almost-reddening}
    Let $Q$ be an abundant acyclic and uniformly sign-coherent quiver, where $v_1 \prec v_2 \prec \dots \prec v_n$ is its unique acyclic ordering.
    If $v_n$ is red in $Q$, then $\mu_{\bw_n}(Q)$ has $n-1$ red vertices and one green vertex $v_n$ for $\bw_n = [v_1, v_2, \dots, v_n]$.
\end{Cor}

Lemma \ref{lem-color-source-acyclic-mutation-order} combines with Lemma \ref{lem-source-fork-mutation-order} to control the color of $r$ in $F^j$.

\begin{Cor} \label{cor-color-source-fork-mutation}
    Let $F$ be a uniformly sign-coherent fork with point of return $r$, where $v_1 \prec v_2 \prec \dots \prec v_{n-1}$ is the unique acyclic ordering on $F \setminus \{r\}$.
    Additionally, let $F^j = \mu_{\bw_j}(F)$ for $\bw_j = [v_1, v_2, \dots, v_j]$ whenever $j \leq n-1$.
    \begin{itemize}
        \item If $r$ is red, then it can be any color in $F^j$.

        \item If $r$ is green, then it remains green in $F^j$.

        \item If $r$ is blue, then it is only blue in $F^j$ if no vertex $v_i$ is green in $F^{i-1}$ for all $i \leq j$.
        Otherwise, it will become green in $F^j$.
    \end{itemize}
    Furthermore, if $v_{n-1}$ is red in $F$, then it is both green in $F^{n-1}$ and the only non-red vertex in $F^{n-1} \setminus \{r\}$.
\end{Cor}

From Corollary \ref{cor-color-source-fork-mutation}, if the sink of $F \setminus \{r\}$ is red, then we can readily mutate to a fork with at most 2 green vertices.
However, we do not need this restriction, as any color will do.

\begin{Lem} \label{lem-usc-fork-mutate-green-por}
    Let $F$ be a uniformly sign-coherent fork with point of return $r$, where $v_1 \prec v_2 \prec \dots \prec v_{n-1}$ is the unique acyclic ordering on $F \setminus \{r\}$.
    Then there is a mutation sequence of length at most $n+4$ ending in a fork with a green point of return, at most one other green vertex, and at least $n-2$ red vertices.
\end{Lem}

\begin{proof}
    Suppose first that $v_{n-1}$ is red.
    If $r$ is red or blue, mutate at $v_{n-1}$.
    This will produce a fork $F'$ with a green point of return $v_{n-1}$ and a red sink $r$ in $F' \setminus \{v_{n-1}\}$ by Lemma \ref{lem-fork-mutate-order}.
    If $r$ is instead green, we already have a fork with a green point of return and a red sink.
    We may then apply Corollary \ref{cor-color-source-fork-mutation} to get the desired result in at most $n$ mutations in both cases.

    Assume now that $v_{n-1}$ is green.
    If $r$ is red, then mutate at $v_{n-1}$.
    This will produce a fork $F'$ with a red point of return $v_{n-1}$ and a red sink $r$ in $F' \setminus \{v_{n-1}\}$.
    If $r$ is green or blue, then mutating $F$ along $[v_{n-1},r]$ will produce another fork $F^*$ with a red or blue point of return $r$ and a red sink $v_{n-1}$ in $F^* \setminus \{r\}$.
    Both of these reduce to our previous case, netting a max of $n+2$ total mutations.

    Assume that $v_{n-1}$ is blue.
    If $r$ is red or green, mutate at $v_{n-1}$.
    This will produce a fork $F'$ with a blue point of return $v_{n-1}$ and a red or green sink $r$ in $F' \setminus \{v_{n-1}\}$.
    Either situation reduces to one of our previous cases, giving us a max of $n+3$ total mutations.
    
    Finally, if $r$ and $v_{n-1}$ are both blue, there must be some other vertex $v_j \prec v_{n-1}$ in $F$ that is not blue by our assumption on uniformly sign-coherent quivers.
    Mutate at $v_j$.
    Let $F^*$ be the fork obtained after mutating at $v_j$.
    Then $F^*$ has a red or green point of return and one of $v_{n-1}$ or $r$ is the sink of $F^* \setminus \{v_j\}$ by Lemma \ref{lem-fork-mutate-order}.
    This gives us one of the three previous cases, so we have arrived at a max of $n+4$ total mutations.
\end{proof}

As we only need strictly sign-coherent, we have the following corollary. 

\begin{Cor} \label{cor-tsc-fork-mutate-green-por}
    Let $F$ be a strictly sign-coherent fork with point of return $r$, where $v_1 \prec v_2 \prec \dots \prec v_{n-1}$ is the unique acyclic ordering on $F \setminus \{r\}$.
    Then there is a mutation sequence of length at most $n+2$ ending in a fork with a green point of return, at most one other green vertex, and at least $n-2$ red vertices.
\end{Cor}

Given an arbitrary fork $F$ on $n$ vertices, Corollary \ref{cor-tsc-fork-mutate-green-por} shows $\ured(F) \geq n-2$.
All that remains is to chip away at one of remaining green vertices.

\begin{Lem} \label{lem-green-por-2-green}
    Let $F$ be a strictly sign-coherent fork with point of return $r$, where $v_1 \prec v_2 \prec \dots \prec v_{n-1}$ is the unique acyclic ordering on $F \setminus \{r\}$.
    If $r$ and $v_j$ are the only two green vertices of $F$, then $F$ has a mutation sequence---of length at most $j+2$---ending in a fork with at least $n-1$ red vertices.
\end{Lem}
    
\begin{proof}
    If $j = 1$, then mutate at $v_1$ to arrive at a fork with a single green vertex $r$.
    Suppose then that $j > 1$.
    Since $v_j$ was the first green vertex in the acyclic ordering of $F \setminus \{r\}$, we know that $v_{j-1}$ is red in $F$.
    By Corollaries \ref{cor-source-acyclic-almost-reddening} and \ref{cor-color-source-fork-mutation}, we know that the vertices $v_{j-1}$, $v_j$, and $r$ are the only possibly green vertices of $F^{j-1} = \mu_{\bw_{j-1}}(F)$ for $\bw_{j-1} = [v_1, \dots, v_{j-1}]$.
    As $r$ is green, mutating at $r$ produces a fork $F'$ with red point of return $r$ and green $v_{j-1}$, where $r \to v_{j-1}$ as a consequence of Corollary \ref{cor-source-r-fork-mutation-order}.
    Since $r$ was a source in $F^{j-1} \setminus \{v_{j-1}\}$, we know that $v_{j-1}$ and $v_{j}$ are the only two possibly green vertices in $F'$.
    If $v_{j}$ is red, then we have arrived at the general reddening sequence $\bw_{j-1}[r]$, as $v_{j-1}$ must still be green.
    If $v_{j}$ is green, then mutate at $v_{j-1}$.
    Let $\Tilde{F} = \mu_{v_{j-1}}(F')$.
    As the acyclic ordering of $F' \setminus \{r\}$ is $v_{j-1} \prec v_{j} \prec  \dots \prec v_{n-1} \prec v_1 \prec \dots \prec v_{j-2}$ by Corollary \ref{cor-source-r-fork-mutation-order}, this will make $v_{j}$ a green source of $\Tilde{F} \setminus \{r\}$.
    Mutation at $v_{j-1}$ will also swap $r$ from red to green and $v_{j-1}$ from green to red, meaning $r$ and $v_j$ are the only green vertices.

    \begin{figure}[ht]
        \centering
        \begin{tikzcd}
             & w\\
            r & & v_{j-1}
            \arrow[from=2-1,to=1-2, "b"]
            \arrow[from=2-3,to=2-1, "a"]
            \arrow[from=2-3,to=1-2, "c"']
        \end{tikzcd}    
        \caption{Subquiver of $F^{j-1}$ Induced by $\{r,v_{j-1},w\}$}
        \label{fig-lem-f-j-1}
    \end{figure}
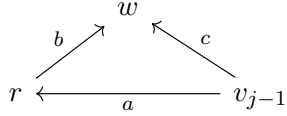
    
    To see why, observe the subquiver of $F^{j-1}$ given in Figure \ref{fig-lem-f-j-1} for some frozen vertex $w$ that touches $r$, $a \geq 2$, $b \geq 1$, and $c \geq 0$.
    Mutation at $r$ produces the subquiver of $F'$ given in Figure \ref{fig-lem-f-prime}.
    As $a > 1$, further mutation at $v_{j-1}$ gives the subquiver of $\Tilde{F}$ given in Figure \ref{fig-lem-f-tilde}, where $a^2b+ac-b > 0$.
    Hence, the vertex $r$ must be green and $v_{j-1}$ must be red in $\Tilde{F}$.

    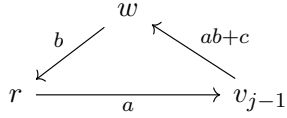
\begin{figure}[ht]
        \centering
            \begin{tikzcd}
             & w\\
            r & & v_{j-1}
            \arrow[from=1-2,to=2-1, "b"']
            \arrow[from=2-1,to=2-3, "a"']
            \arrow[from=2-3,to=1-2, "ab+c"']
        \end{tikzcd}
        \caption{Subquiver of $F'$ Induced by $\{r,v_{j-1},w\}$}
        \label{fig-lem-f-prime}
    \end{figure}

    Finally, as $v_j$ is a green source of $\Tilde{F} \setminus \{r\}$, we mutate at $v_{j}$ to arrive at a fork with a singular green vertex $r$.
    This completes the mutation sequence of length at most $j+2$.

    \begin{figure}[ht]
        \centering
        \begin{tikzcd}
             & w\\
            r & & v_{j-1}
            \arrow[from=2-1,to=1-2, "a^2b+ac-b"]
            \arrow[from=2-3,to=2-1, "a"]
            \arrow[from=1-2,to=2-3, "ab+c"]
        \end{tikzcd}    
        \caption{Subquiver of $\Tilde{F}$ Induced by $\{r,v_{j-1},w\}$}
        \label{fig-lem-f-tilde}
    \end{figure}
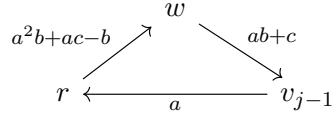
    
\end{proof}

We may then prove our main theorem.

\begin{Thm} \label{thm-tsc-fork-almost-reddening}
    Let $F$ be a strictly sign-coherent fork on $n$ mutable vertices.
    Then there exists a mutation sequence of length at most $2n+3$ ending in a fork with at least $n-1$ red vertices.
\end{Thm}

\begin{proof}
    By Corollary \ref{cor-tsc-fork-mutate-green-por}, it takes at most $n+2$ mutations to arrive at a fork with a green point of return and at most one other green vertex.
    By Lemma \ref{lem-green-por-2-green}, once we have arrived at such a fork, it takes at most $n+1$ mutations to find a fork with $n-1$ red vertices and $1$ green vertex.
    As such, there exists a mutation sequence of length at most $2n+3$ ending in a fork with at least $n-1$ red vertices for every strictly sign-coherent fork on $n$ vertices.
\end{proof}

For any connected, mutation-infinite quiver $Q$, we may take the framing $\widehat{Q}$.
Then Theorem \ref{thm-connected-fork} tells us that we can mutate $\widehat{Q}$ to a strictly sign-coherent fork.
Our main result follows.

\begin{Cor} \label{cor-tsc-fork-ured-size}
    Every connected, mutation-infinite quiver $Q$ on $n$ vertices has $uRed(Q) \geq n-1$.
    Furthermore, if $Q$ is any quiver on $n$ vertices, then $\ured(Q) = n - c$, where $c$ is the number of connected components of $Q$ that do not admit reddening sequences.
\end{Cor}

\section{Controlling c-vectors} \label{sec-sign-coherent-forks}

While observing strictly sign-coherent forks, a pattern began to emerge in the number and direction of arrows between mutable and frozen vertices.
As such, we looked at the $c$-vectors arising from mutation sequences that never mutated at a fork's point of return.
Proposition \ref{prop-control-signs-fork} is the result, and it controls the signs and magnitudes of the $c$-vectors as long as certain conditions on the starting ice quiver are met.
A future work will make use of this result \cite{ervin_geometry_2024}.
However, we first demonstrate that a broad class of ice quivers satisfy the assumptions necessary for Proposition \ref{prop-control-signs-fork}.

\begin{Lem} \label{lem-framings-sign-coherent}
    Let $Q$ be any ice quiver with all red or all green vertices.
    If $v$ is a vertex of $Q$, then
    \begin{itemize}
        \item Every mutable vertex of $\mu_v(Q)$ is either red or green;
        
        \item If $b_{vj}^{[v]} c_j^{[v]} \geq 0$ and $b_{vj}^{[v]} \neq 0$ for some vertex $j \neq v$, then either $\sgn(c_j^{[v]}) = \sgn(c_v^{[v]})$ or $ \sgn(c_j^{[v]}) c_j^{[v]} \geq \sgn(c_v^{[v]}) c_v^{[v]}$;

        \item If $b_{vj}^{[v]} c_j^{[v]} \leq 0$ for some vertex $j \neq v$, then $\sgn(c_j^{[v]}) = -\sgn(c_v^{[v]})$;
        
        \item If $b_{ij}^{[v]} c_j^{[v]} \geq 0$ for distinct vertices $i$ and $j$ that are neither $v$, then $\sgn(c_i^{[v]}) = \sgn(c_j^{[v]})$.
    \end{itemize}
    In particular, the framing, coframing, and any triangular extension \cite{eric_bucher_banff_2021} of a mutable quiver with frozen vertices satisfy the above.
\end{Lem}

\begin{proof}
    First, every mutable vertex other than $v$ of $\mu_v(Q)$ retains their color, as mutating at $v$ can only add arrows between frozen and mutable vertices in the same direction as they were already pointing.
    As mutating at $v$ only reverses the direction of the arrows at $v$, it is now the opposite color of what it was originally, proving the first result.
    In terms of $c$-vectors, we have that $c_v^{[v]} = -c_v$ and $\sgn(c_i^{[v]}) = \sgn(c_i)$ for all vertices $i \neq v$.
    
    Thus $\sgn(c_v^{[v]}) = -\sgn(c_j^{[v]})$ and $\sgn(c_i^{[v]}) = \sgn(c_j^{[v]})$ for all distinct vertices $i$ and $j$ with $v \notin \{i,j\}$, demonstrating the latter two points.
    Finally, if $b_{vj}^{[v]} c_j^{[v]} \geq 0$ and $b_{vj}^{[v]} \neq 0$ for some vertex $j \neq v$, then $b_{jv}^{[v]} c_v^{[v]} \geq 0$ as $\sgn(c_v^{[v]}) = -\sgn(c_j^{[v]})$.
    Thus $b_{jv} c_v \geq 0$ and 
    $$c_j^{[v]} = c_j + |b_{jv}|c_v.$$
    As $c_j$ and $c_v$ share the same sign and $b_{jv} \neq 0$, we naturally have that  $ \sgn(c_j^{[v]}) c_j^{[v]} \geq \sgn(c_v^{[v]}) c_v^{[v]}$, completing our proof.
\end{proof}

As many ice quivers satisfy the assumptions of Proposition \ref{prop-control-signs-fork}, we may now proceed with the inductive step of the proof.

\begin{Prop} \label{prop-control-signs-fork}
    Let $Q$ be any ice quiver such that the mutable subquiver of $\mu_v(Q)$ is a fork with point of return $v$ for some vertex $v$.
    Further, let $\bw$ be any reduced mutation sequence beginning with $v$ and suppose that
    \begin{itemize}
        \item Every mutable vertex of $\mu_v(Q)$ is either red or green;
        
        \item If $b_{vj}^{[v]} c_j^{[v]} \geq 0$ for some vertex $j \neq v$, then either $\sgn(c_j^{[v]}) = \sgn(c_v^{[v]})$ or $ \sgn(c_j^{[v]}) c_j^{[v]} \geq \sgn(c_v^{[v]}) c_v^{[v]}$;

        \item If $b_{vj}^{[v]} c_j^{[v]} \leq 0$ for some vertex $j \neq v$, then $\sgn(c_j^{[v]}) = -\sgn(c_v^{[v]})$;
        
        \item If $b_{ij}^{[v]} c_j^{[v]} \geq 0$ for distinct vertices $i$ and $j$ that are neither $v$, then $\sgn(c_i^{[v]}) = \sgn(c_j^{[v]})$.
    \end{itemize}
    Then, if $r$ is the last mutation of $\bw$, we have the following:
    \begin{itemize}
        \item Every mutable vertex of $\mu_\bw(Q)$ is either red or green;
        
        \item If $b_{rj}^\bw c_j^\bw \geq 0$ for some vertex $j \neq r$, then either $\sgn(c_j^{\bw}) = \sgn(c_r^\bw)$ or $ \sgn(c_j^\bw) c_j^\bw \geq \sgn(c_r^\bw) c_r^\bw$;

        \item If $b_{rj}^\bw c_j^\bw \leq 0$ for some vertex $j \neq r$, then $\sgn(c_j^{\bw}) = -\sgn(c_r^\bw)$;
        
        \item If $b_{ij}^\bw c_j^\bw \geq 0$ for distinct vertices $i$ and $j$ that are neither $r$, then $\sgn(c_i^{\bw}) = \sgn(c_j^\bw)$.
    \end{itemize}
\end{Prop}

\begin{proof}
    Take for our inductive hypothesis that the four conclusions hold for a reduced mutation sequence $\bw'$ beginning with $v$ of length $m \geq 1$.
    Our base case was managed by our assumptions on the mutation sequence $[v]$ and $\mu_v(Q)$.
    If $\bw = \bw'[k]$ for some mutation sequence $\bw'$ of length $m$ ending in $r$ and some vertex $k \neq r$, then we will show that the same conclusions hold for $\bw$ and $\mu_\bw(Q)$.

    Assume for now that $i \notin \{r,k\}$.
    By the inductive hypothesis, we know $b_{ik}^{\bw'} c_k^{\bw'} \geq 0$ implies $\sgn(c_i^{\bw'}) = \sgn(c_k^{\bw'})$.
    Then $c_i^{\bw} = c_i^{\bw'} + |b_{ik}^{\bw'}| c_k^{\bw'}$, which gives us
    $$\sgn(c_i^{\bw}) = \sgn(c_i^{\bw'}) = \sgn(c_k^{\bw'}) = -\sgn(c_k^{\bw}).$$
    The magnitude of each entry of $c_i^{\bw}$ is then greater than or equal to the magnitude of each entry of $c_k^{\bw} = -c_k^{\bw'}$ as $|b_{ik}^{\bw'}| > 0$ from $\mu_{\bw'}(Q)$ being a fork.
    Thus $b_{ki}^{\bw} c_i^{\bw} \geq 0$ and $\sgn(c_i^\bw) c_i^\bw \geq \sgn(c_k^\bw) c_k^\bw$, as desired.

    If instead $b_{ik}^{\bw'} c_k^{\bw'} \leq 0$, we know that $c_i^{\bw} = c_i^{\bw'}$.
    As such, if $b_{ki}^{\bw} c_i^{\bw} \geq 0$, then $b_{ki}^{\bw'} c_i^{\bw'} \leq 0$.
    This forces 
    $$\sgn(c_i^{\bw}) = \sgn(c_i^{\bw'}) = -\sgn(c_k^{\bw'}) = \sgn(c_k^{\bw}).$$
    If $b_{ki}^{\bw} c_i^{\bw} \leq 0$, then $b_{ki}^{\bw'} c_i^{\bw'} \geq 0$.
    Thus the inductive hypothesis gives
    $$\sgn(c_i^{\bw}) = \sgn(c_i^{\bw'}) = \sgn(c_k^{\bw'}) = -\sgn(c_k^{\bw}).$$
    Both of the above situations are the anticipated outcomes.
    Additionally, they combine to show that $\sgn(c_i^\bw) > 0$ or $\sgn(c_i^\bw) < 0$ for all $i \neq r$.
    This is equivalent to every vertex other than possibly $r$ being red or green.

    Now we tackle the case where $i = r$.
    If $b_{rk}^{\bw'} c_k^{\bw'} \leq 0$, then $\sgn(c_k^{\bw'}) = - \sgn(c_r^{\bw'})$ by the inductive hypothesis.
    As such, we have that $c_r^{\bw} = c_r^{\bw'}$, $\sgn(c_r^{\bw}) = \sgn(c_k^{\bw})$, and $b_{kr}^{\bw} c_r^{\bw} \geq 0$, as desired.
    If $b_{rk}^{\bw'} c_k^{\bw'} \geq 0$, then either $\sgn(c_k^{\bw'}) =  \sgn(c_r^{\bw'})$ or $\sgn(c_k^{\bw'})c_k^{\bw'} \geq  \sgn(c_r^{\bw'})c_r^{\bw'}$ by the inductive hypothesis.
    In the first case, we know that $c_r^{\bw} = c_r^{\bw'} + |b_{rk}^{\bw'}| c_k^{\bw'}$ has the same sign as $c_k^{\bw'}$ and thus the opposite sign of $c_k^{\bw}$.
    This forces $b_{kr}^{\bw} c_r^{\bw} \geq 0$ and $\sgn(c_r^{\bw}) c_r^{\bw} \geq \sgn(c_k^{\bw}) c_k^{\bw}$.
    
    When $\sgn(c_k^{\bw'}) =  -\sgn(c_r^{\bw'})$ and $\sgn(c_k^{\bw'})c_k^{\bw'} \geq  \sgn(c_r^{\bw'})c_r^{\bw'}$, this again forces $c_r^{\bw} = c_r^{\bw'} + |b_{rk}^{\bw'}| c_k^{\bw'}$ to have the opposite sign of $c_k^{\bw}$.
    Furthermore, our assumption gives 
    $$\sgn(c_r^{\bw}) c_r^{\bw} = \sgn(c_k^{\bw'}) [c_r^{\bw'} + |b_{rk}^{\bw'}| c_k^{\bw'}]$$
    $$= \sgn(c_k^{\bw'}) c_r^{\bw'} + \sgn(c_k^{\bw'})c_k^{\bw'} + \sgn(c_k^{\bw'})(|b_{rk}^{\bw'}|-1)c_k^{\bw'}$$
    $$\geq \sgn(c_k^{\bw'}) c_r^{\bw'} + \sgn(c_r^{\bw'})c_r^{\bw'} + \sgn(c_k^{\bw'})(|b_{rk}^{\bw'}|-1)c_k^{\bw'}$$
    $$= \sgn(c_k^{\bw'})(|b_{rk}^{\bw'}|-1)c_k^{\bw'}$$
    $$\geq  \sgn(c_k^{\bw'}) c_k^{\bw'}$$
    $$\geq  \sgn(c_k^{\bw}) c_k^{\bw}$$
    with the last inequality coming from the abundant property of forks.
    This completes inductive step for the first two sign conditions, and it shows that every vertex is either red or green.
    We only need to demonstrate the last sign condition.
    Assume then that $b_{ij}^{\bw} c_j^{\bw} \geq 0$ for distinct vertices $i$ and $j$ that are neither $k$ or $r$.
    The case where $r \in \{i,j\}$ will be treated as a special case at the end of the proof.

    First, note that $\sgn(b_{ij}^{\bw}) = \sgn(b_{ij}^{\bw'})$ if $r \notin \{i,j\}$, as $B^{\bw'} \setminus \{r\}$ is an acyclic quiver and we mutated at $k \notin\{i,j, r\}$.
    Thus, if $b_{jk}^{\bw'} c_k^{\bw'} \leq 0$, we have that $b_{ij}^{\bw'} c_j^{\bw'} \geq 0$ as $c_j^{\bw} = c_j^{\bw'}$.
    This forces $\sgn(c_i^{\bw'}) = \sgn(c_j^{\bw'})$ by the inductive hypothesis.
    If $\sgn(c_k^{\bw'}) = \sgn(c_j^{\bw'})$, then 
    $$\sgn(c_i^{\bw}) = \sgn(c_i^{\bw'}) = \sgn(c_j^{\bw'}) = \sgn(c_j^{\bw}).$$ 
    If $\sgn(c_k^{\bw'}) = -\sgn(c_j^{\bw'})$, then $b_{ik}^{\bw'} c_k^{\bw'} \leq 0$.
    Otherwise, the inductive hypothesis would imply that $\sgn(c_i^{\bw'}) = \sgn(c_k^{\bw'}) = -\sgn(c_j^{\bw'})$, a contradiction.
    Thus
    $$\sgn(c_i^{\bw}) = \sgn(c_i^{\bw'}) = \sgn(c_j^{\bw'}) = \sgn(c_j^{\bw}).$$
    In both cases, we have $b_{ij}^{\bw} c_j^{\bw} \geq 0$ and $\sgn(c_i^{\bw}) = \sgn(c_j^{\bw})$.
    
    Now, if $b_{jk}^{\bw'} c_k^{\bw'} \geq 0$, the inductive hypothesis gives that $\sgn(c_j^{\bw'}) = \sgn(c_k^{\bw'})$.
    Since $\sgn(c_j^{\bw}) = \sgn(c_j^{\bw'})$ and $\sgn(b_{ij}^{\bw}) = \sgn(b_{ij}^{\bw'})$, we again have $b_{ij}^{\bw'} c_j^{\bw'} \geq 0$.
    The inductive hypothesis then gives that $\sgn(c_i^{\bw'}) = \sgn(c_j^{\bw'})$.
    Thus
    $$\sgn(c_i^{\bw}) = \sgn(c_i^{\bw'}) = \sgn(c_k^{\bw'}) = \sgn(c_j^{\bw'}) = \sgn(c_j^{\bw}),$$
    completing the last sign condition for the case where $r \notin \{i,j\}$.

    Finally, if $b_{ij}^{\bw} c_j^{\bw} \geq 0$ and $\sgn(c_i^{\bw}) = -\sgn(c_j^{\bw})$ for distinct vertices $i$ and $j$ that are neither $k$ with $r \in \{i,j\}$, we will soon arrive at a contradiction.
    As this forces $b_{ji}^{\bw} c_i^{\bw} \geq 0$, we can assume without loss of generality that $i = r$, i.e., $b_{rj}^\bw c_j^\bw \geq 0$ and $b_{jr}^\bw c_r^\bw \geq 0$.
    Then 
    $$-\sgn(b_{jr}^{\bw}) = \sgn(b_{rj}^{\bw}) = \sgn(b_{rk}^{\bw'}) = \sgn(b_{kr}^{\bw})$$ 
    as mutation at $k$ will force $b_{rv}^\bw$ to have the sign of $b_{rk}^{\bw'}$ for all vertices $v \notin \{k,r\}$.
    Thus $b_{kr}^\bw c_r^\bw \leq 0$, forcing $\sgn(c_r^\bw) = - \sgn(c_k^\bw)$ as shown previously.
    Then both $b_{rk}^{\bw} c_k^{\bw} \leq 0$ and $b_{rk}^{\bw'} c_k^{\bw'} \leq 0$, forcing $c_r^{\bw} = c_r^{\bw'}$.
    By the inductive hypothesis, we know that $\sgn(c_r^{\bw'}) = -\sgn(c_k^{\bw'})$.
    However, this creates a contradiction as 
    $$\sgn(c_r^\bw) = \sgn(c_r^{\bw'}) = -\sgn(c_k^{\bw'}) = \sgn(c_k^{\bw}).$$
    Therefore, we cannot have both $b_{rj}^{\bw} c_j^{\bw} \geq 0$ and $\sgn(c_r^{\bw}) = - \sgn(c_j^{\bw})$ for some vertex $j \notin \{k,r\}$, completing the inductive step of our proof.
    Thus the four conclusions hold for any reduced mutation sequence $\bw$ beginning with $v$.
\end{proof}

We then have an immediate corollary on the case where $b_{ir}^{\bw} c_r^\bw \geq 0$.

\begin{Cor} \label{cor-signs-por}
    Let $Q$ be any ice quiver such that the mutable subquiver of $\mu_v(Q)$ is a fork with point of return $v$ for some vertex $v$ satisfying the assumptions of Proposition \ref{prop-control-signs-fork}.
    If $r$ is the last mutation of $\bw$---hence the point of return of $Q^\bw$---and $b_{ir}^{\bw} c_r^\bw \geq 0$ for some vertex $i \neq r$, then $\sgn(c_i^\bw) = -\sgn(c_r^\bw)$ and $\sgn(c_i)^{\bw} c_i^{\bw} \geq \sgn(c_r^\bw) c_r^\bw$. 
\end{Cor}

Additionally, we can easily generalize Proposition \ref{prop-control-signs-fork} to abundant acyclic quivers.

\begin{Lem} \label{lem-control-signs-abundant acyclic}
    Let $Q$ be any ice quiver such that the mutable subquiver of $\mu_v(Q)$ is either a fork with point of return $v$ or an abundant acyclic quiver with $v$ a sink or source.
    Further, let $\bw$ be any reduced sequence beginning with $v$.
    Then, if the assumptions of Proposition \ref{prop-control-signs-fork} hold, so do the conclusions.
\end{Lem}

\begin{proof}
    As abundant acyclic quivers are mutation-abundant, we only need to check the last part of the proof, where we arrived at our contradiction.
    To do so, we need to show that mutation at $k$ will force $b_{rv}^\bw$ to have the sign of $b_{rk}^{\bw'}$ for all vertices $v \notin \{k,r\}$.
    If $\mu_\bw(Q)$ is abundant acyclic, then both $r$ and $k$ were source or sink mutations.
    As $\bw$ is a reduced mutation sequence, to do otherwise would force $\mu_\bw(Q)$ to be a fork.
    Thus, the vertex $r$ was a sink or source in $\mu_{\bw'}(Q)$.
    As such,
    $$-\sgn(b_{jr}^{\bw}) = \sgn(b_{rj}^{\bw}) = \sgn(b_{rj}^{\bw'}) = \sgn(b_{rk}^{\bw'}) = \sgn(b_{kr}^{\bw}),$$ 
    and the proof goes through as before.

    Assume then that $\mu_\bw(Q)$ is a fork.
    In the case that $\mu_{\bw'}(Q)$ is a fork, then the proof works without modification.
    In the case that $\mu_{\bw'}(Q)$ is abundant acyclic, then $r$ is either a source or a sink.
    The above argument then finishes our generalization.
\end{proof}

If we reduce to the case where we have three mutable vertices, we can further generalize to mutation-cyclic quivers on three vertices.

\begin{Lem} \label{lem-control-signs-3-cyclic}
    Let $Q$ be any ice quiver such that the mutable subquiver of $\mu_v(Q)$ is a mutation-cyclic quiver on three vertices.
    Further, let $v$ be any vertex and $\bw$ be any reduced mutation sequence beginning with $v$.
    Then, if the assumptions of Proposition \ref{prop-control-signs-fork} hold, so do the conclusions.
\end{Lem}

\begin{proof}
    As mutation-cyclic quivers on three vertices are mutation-abundant, we need only show that mutation at $k$ will force $b_{rj}^\bw$ to have the sign of $b_{rk}^{\bw'}$ for all vertices $j \notin \{k,r\}$.
    Since our quiver is a cycle and mutating produces a cycle with opposite orientation, we have 
    $$-\sgn(b_{jr}^{\bw}) =  \sgn(b_{rj}^{\bw}) = -\sgn(b_{rj}^{\bw'}) = \sgn(b_{rk}^{\bw'}) = -\sgn(b_{rk}^{\bw}) =  \sgn(b_{kr}^{\bw}),$$
    and the proof follows as before.
\end{proof}

Notably, any framing or coframing of a fork $F$ will have $\mu_v(F)$ satisfy the constraints of Proposition \ref{prop-control-signs-fork} for all vertices $v$ that are not the point of return by Lemma \ref{lem-framings-sign-coherent}.
If instead we look at a framing or coframing of $Q$ for arbitrary abundant acyclic quivers or for mutation-cyclic quivers on three vertices, we see that $\mu_v(Q)$ satisfies the constraints of Proposition \ref{prop-control-signs-fork} for all vertices $v$.
We can then use this information to prove that the latter two types of quivers are strictly sign-coherent. 

\begin{Cor} \label{cor-sign-coherence-abundant acyclic}
    Let $Q$ be any ice quiver such that the mutable subquiver of $Q$ is an abundant acyclic quiver.
    If $\mu_v(Q)$ satisfies the assumptions of Lemma \ref{lem-control-signs-abundant acyclic} for all mutable vertices $v$, then $Q$ is strictly sign-coherent.
    In particular, any framing or coframing of an abundant acyclic quiver is strictly sign-coherent.
\end{Cor}

\begin{Cor} \label{cor-sign-coherence-mutation-cyclic-3}
    Let $Q$ be any ice quiver such that the mutable subquiver of $Q$ is a mutation-cyclic quiver on three vertices.
    If $\mu_v(Q)$ satisfies the assumptions of Lemma \ref{lem-control-signs-abundant acyclic} for all mutable vertices $v$, then $Q$ is strictly sign-coherent.
    In particular, any framing or coframing of a mutation-cyclic quiver on three vertices is strictly sign-coherent.
\end{Cor}

Finally, both the Markov quiver---the cycle on three vertices with two arrows between each pair of vertices---and arbitrary mutation-cyclic quivers on three vertices do not admit a maximal green sequence \cite{muller_existence_2016}.
It was additionally shown that the Markov quiver does not admit a reddening sequence \cite{ladkani_cluster_2013}.
Further, it is well-known that arbitrary mutation-cyclic quivers on three vertices do not admit a reddening sequence; however, we are unaware of where this was first proven in the literature.
We give an elementary proof of this fact using Lemma \ref{lem-control-signs-3-cyclic}.

\begin{Lem} \label{lem-red-seq-3-cycle}
    Let $Q$ be any mutation-cyclic quiver on three vertices.
    Then the framing of $Q$ is not mutation-equivalent to any quiver with all red vertices.
    As such, any mutation-cyclic $Q$ on three vertices does not admit a reddening sequence.
\end{Lem}

\begin{proof}
    We argue by contradiction.
    Assume that we may mutate the framing of $Q$ to an ice quiver $P$ where all the vertices are red.
    Let $r$ be the last mutation of the sequence $\bw$ that takes $\widehat{Q}$ to $P$.
    Since $Q$ is mutation-cyclic and every vertex of $P$ is red, we know that $b_{ri}^\bw c_i^\bw \geq 0$ and $b_{rj}^\bw c_j^\bw \leq 0$ for mutable vertices $i$ and $j$ in $P$.
    However, this would imply that $\sgn(c_r^\bw) = -\sgn(c_j^\bw)$ by Lemma \ref{lem-control-signs-3-cyclic}, contradicting all vertices of $P$ being red.
    Therefore, the framing $\widehat{Q}$ cannot be mutated to any ice quiver with all red vertices, forcing $Q$ to not admit a reddening sequence.
\end{proof}

\begin{Rmk} \label{rmk-red-seq-3-acyclic}
    This argument can also be applied to abundant acyclic quivers on an arbitrary number of vertices.
    Since the vertex $r$ can be a source or sink, the contradiction reached in Lemma \ref{lem-red-seq-3-cycle} can be avoided.
    As such, the only reddening sequences available to abundant acyclic quivers is those that reduce to source mutation sequences.
    Similar exploration of possible reddening sequences for more arbitrary quivers will be presented in an upcoming work with Scott Neville \cite{ervin_mutation_2024}.
\end{Rmk}

\bibliography{references}
\bibliographystyle{plain}

\end{document}